\documentclass{amsart}
\usepackage{amsmath}
  \usepackage{paralist}

\newtheorem{theorem}{Theorem}[section]
\newtheorem{corollary}{Corollary}

\newtheorem{lemma}[theorem]{Lemma}
\newtheorem{proposition}{Proposition}

\theoremstyle{definition}

\newtheorem{remark}{Remark}

\def\N{\mathbb{N}}
\def\R{\mathbb{R}}
\def\io{\int\limits_{\Omega}}
\def\div{{\rm div}}

\title[An elliptic problem with
degenerate coercivity]{
An elliptic problem with
degenerate coercivity
and a singular quadratic gradient lower order term
}
\author[Gisella Croce]{}

\subjclass{35J15, 35J25, 35J66, 35J70, 35J75}
 \keywords{Nonlinear elliptic problems, Dirichlet condition, degenerate coercivity, distributional solution, singular lower order term, quadratic growth}

 \email{gisella.croce@univ-lehavre.fr}
 
\thanks{}

\begin{document}
\maketitle

\centerline{\scshape  Gisella Croce}
\medskip
{\footnotesize
 \centerline{Laboratoire de Math\'ematiques Appliqu\'ees du Havre  }
   \centerline{Universit\'e du Havre}
   \centerline{25, rue Philippe Lebon}
   \centerline{76063 Le Havre (FRANCE)}
} 

\medskip

\bigskip


\begin{abstract}
In this paper we study a Dirichlet problem for an elliptic equation with 
degenerate coercivity and a singular lower order term with natural growth with respect to the gradient.
The model problem is
$$
\left\{
\begin{array}{cl}
\displaystyle -\div\left(\frac{\nabla u}{(1+|u|)^p}\right) + \frac{|\nabla u|^{2}}{|u|^{\theta}} = f & \mbox{in $\Omega$,} \\
\hfill u = 0 \hfill & \mbox{on $\partial\Omega$,}
\end{array}
\right.
$$
where $\Omega$ is an open bounded set of $\R^N$, $N\geq 3$ and  $p, \theta>0$.
The source $f$ is a positive function belonging to some  Lebesgue space.
We will show that, even if the lower order term is singular, it has some regularizing effects on the solutions, when $p>\theta-1$ and $\theta<2$.
\end{abstract}

\section{Introduction}

In this paper we study the following problem:
\begin{equation}\label{pb}
\begin{cases}
\displaystyle
-\div\left(\frac{b(x)}{(1+|u|)^p}\nabla u\right) + B\frac{|\nabla u|^{2}}{|u|^{\theta}} = f & \mbox{in $\Omega$,} \\
\hfill u = 0 \hfill & \mbox{on $\partial\Omega$,}
\end{cases}
\end{equation}
where $\Omega$ is an open bounded set of $\R^N$, $N\geq 3$,  $B, p>0$ and  $\theta>0$.
We assume that  $b: \Omega \to \R$ is a measurable function such that for some positive constants
$\alpha$ and $\beta$ 
\begin{equation}\label{hyp_a}
{\alpha}\leq b(x)\leq \beta\,\,\,\,\,\,\,\,\mbox{for  a.e.}\, x \in \Omega\,.
\end{equation}
Moreover $f$ is a positive function belonging to some  Lebesgue space $L^m(\Omega)$, with $m\geq 1$.
 We point out three characteristics of this problem: 
the operator $\displaystyle A(v)=-\div\left(\frac{b(x)}{(1+|v|)^p}\nabla v\right)$ is defined on $H^{1}_0(\Omega)$ 
but is not coercive on this space when $v$ is large,  as proved in \cite{porretta_degenerate}. 
The lower order term has a quadratic growth with respect to the gradient and is singular in the variable $u$. 
As we will see, existence and summability of solutions to problem (\ref{pb}) depend on these features.

It is known that the degenerate coercivity has in some sense a bad effect on the summability of the solutions to
problem
\begin{equation}\label{without}
\begin{cases}
\displaystyle
-\div\left(a(x,u){\nabla u}\right) = f & \mbox{in $\Omega$,} \\
\hfill u = 0 \hfill & \mbox{on $\partial\Omega$,}
\end{cases}
\end{equation}
as proved in \cite{bdo}.
There  $f \in L^m(\Omega)$ was not assumed to be positive, $a:\Omega\times \R\to \R$ was a Carath\'eodory function such that
$\displaystyle \frac{\alpha}{(1+|s|)^p}\leq a(x,s)\leq \beta$, for $p \in (0,1)$ and $\alpha, \beta>0$.
Apart from the case where $\displaystyle m>\frac{N}{2}$, the summability of the solutions is lower than the summability of the solutions
to elliptic coercive problems.
Indeed, in \cite{bdo} it is shown that if $\displaystyle \frac{2N}{N+2-p(N-2)}<m<\frac{N}{2}$ there exists a $H^1_0(\Omega)\cap L^r(\Omega)$ distributional solution,  with $\displaystyle r=\frac{Nm(1-p)}{N-2m}$;
if $\displaystyle \frac{N}{N+1-p(N-1)}<m<\frac{2N}{N+2-p(N-2)}$, there exists a  $W^{1,s}_0(\Omega)$ distributional solution, 
with
$\displaystyle s=\frac{Nm(1-p)}{N-m(1+p)}$. For $p> 1$ the authors prove a non-existence result for constant sources $f$. 
Note that a bad effect on the regularity of the solutions appears even when the right hand side of (\ref{without})
is an element of $H^{-1}(\Omega)$, such as $-\div (F)$, with $F \in L^2(\Omega)$. As a matter of fact, in this case the solutions are in general  not in 
$H^1_0(\Omega)$ (see \cite{GP2001}).

The presence of lower order terms can have a regularizing effect on the solutions.
In \cite{boccardo_nonlinear_convex_analysis} and \cite{croce}
 three kinds of  lower order terms are considered for elliptic problems with degenerate coercivity, with no restriction on $p$.
In the first paper the author analyses a lower order term defined by a Carath\'eodory function 
$g: \Omega \times \R\times \R^N\to \R^N$ with the following properties. There exists $d\in L^1(\Omega)$,  
two positive constants
$\mu_1, \mu_2>0$ and a continuous increasing real function $h$ such that
$g(x,s,\xi)s\geq 0$, $\mu_1 |\xi|^2\leq |g(x,s,\xi)|$  when $|s|\geq \mu_2$ and $|g(x,s,\xi)|\leq d(x)h(|s|)|\xi|^2$.
It is proved  that 
for a   $L^1(\Omega)$ source there exists a $H^1_0(\Omega)$ distributional solution to
$$
\begin{cases}
\displaystyle
-\div\left(a(x,u){\nabla u}\right)+ g(x,u,\nabla u)= f & \mbox{in $\Omega$,} \\
\hfill u = 0 \hfill & \mbox{on $\partial\Omega$\,.}
\end{cases}
$$
This proves that the summability of the gradient of the solutions is much larger than that one of the solutions of problem (\ref{without}).
It is even larger than the summability of the gradient of the solutions to elliptic coercive problems with $L^1(\Omega)$ sources, which is $L^s(\Omega)$ for every $s<\frac{N}{N-1}$ (see \cite{b-g} for example).
We remark moreover that the lower order term gives the existence of a solution for $p\geq 1$;  for these values of $p$, (\ref{without})
has no solution.

In a previous article  \cite{croce} we consider two kinds of lower order terms $h(u)$.
For $h(u)=|u|^{q-1}u$, with $q>p+1$, we stablish the  existence of
a distributional solution $\displaystyle u \in W^{1,t}_0(\Omega)\cap L^q(\Omega)\,,\,\,\,t<\frac{2q}{p +1+q}$,  for any $L^1(\Omega)$ source $f$. 
If $f \in L^m(\Omega), m>1$
and $\displaystyle q \geq  \frac{p +1}{m-1}$
then there exists a distributional solution $u$ in $H^1_0(\Omega) \cap L^{qm}(\Omega)\,.$
If $\displaystyle \frac{p+1}{2m-1}< q <  \frac{p +1}{m-1}$, 
 there exists a distributional solution $u$ in 
$W^{1,\frac{2qm}{p +1+q}}_0(\Omega)$
such that 
$|u|^{qm} \in L^{1}(\Omega)$. 
These results show that if $q$ is sufficiently large, there exists a distributional solution for any source; this is not the case 
for problem (\ref{without}).
The second lower order term analysed in \cite{croce} is $h(u)$, where $h: [0, s_0)\to \R$
is a continuous, increasing function such that $h(0)=0$ and $\lim\limits_{s\to s_0^-} h(s)=+\infty$
for some $s_0>0$.
The regularizing effects of this lower order term are even better than the previous one. Indeed
for a  positive $L^1(\Omega)$ source,  there exists a bounded $H^1_0(\Omega)$
solution.

In the literature we find several papers about  elliptic coercive problems 
with  lower order terms having a quadratic growth with respect to the gradient (see \cite{bensoussan, b-g, bmp, bpm_siam, boccardo_per_puel} 
for example and the references therein), that is, for problem 
$$
\begin{cases}
\displaystyle
-\div(M(x)\nabla u) + g(u)|\nabla u|^2= f & \mbox{in $\Omega$,} \\
\hfill u = 0 \hfill & \mbox{on $\partial\Omega$.}
\end{cases}
$$
In these works it is assumed that $M: \Omega \to \R^{N^2}$ is a bounded elliptic Carath\'eodory map, so that there exists $\alpha>0$ such that 
$\alpha |\xi|^2\leq M(x)\xi\cdot \xi$ for every $\xi \in \R^N$.
Various assumptions are made on $g$. With no attempt of being exhaustive,  
we will describe some recent results
where a singular $g$  has been considered, namely $\displaystyle g(u)=\frac{1}{|u|^\theta}$. 
The case where  $0<\theta\leq 1$, introduced in \cite{arcoya1, arcoya2, arcoya3}, 
has been studied in 
\cite{arcoya1, arcoya2, arcoya3, boccardo_per_puel, bop, giachetti-murat}. 
From this body of literature one can deduce that for a positive source $f\in L^m(\Omega)$, 
if $\displaystyle \frac{2N}{2N-\theta(N-2)}\leq m<\frac N2$  
there exists a strictly positive  solution $\displaystyle u\in H^1_0(\Omega)\cap L^{(2-\theta)m^{**}}(\Omega)$; if $\displaystyle 1<m<\frac{2N}{2N-\theta(N-2)}$
then the solution $u$ belongs to $W^{1,q}_0(\Omega)$, with $\displaystyle q=\frac{Nm(2-\theta)}{N-m\theta}$.
The authors of \cite{luigietaltri} consider the general case $\theta< 2$, assuming 
that $f$ is a strictly positive function on every compactly contained subset of $\Omega$. 
They prove that if $f \in L^{\frac{2N}{N+2}}(\Omega)$ 
there exists a positive $H^1_0(\Omega)$ solution.
Finally, in \cite{giachetti-murat} the lower order term is taken to be  $\displaystyle \lambda u+\mu \frac{|\nabla u|^2}{|u|^\theta}\chi_{\{u>0\}}$, where
$\chi_{\{u>0\}}$ denotes the characteristic function of the set $\{u>0\}$, $\lambda>0$ and $\mu \in \R$.

In this paper we consider the same lower order term as above in an  elliptic problem defined by an operator with degenerate coercivity. We will see that if $0<\theta<2$, then $\displaystyle \frac{|\nabla u|^2}{|u|^\theta}$ has a regularizing effect, even if it is singular in $u$. 
We are going to state our results. We will distinguish the cases $0<\theta<1$ and $1\leq \theta<2$. 

\begin{theorem}\label{teorema_energia_finita}
Let $0<\theta<1$.
Assume that $f$ is a positive function belonging to $L^m(\Omega)$, with $\displaystyle m \geq \frac{2N}{2N-\theta (N - 2)}$.
Then there exists a function $u \in H^1_0(\Omega)$, strictly positive on $\Omega$, such that $\displaystyle \frac{|\nabla u|^2}{u^{\theta}} \in L^1(\Omega)$ and
\begin{equation}\label{formulazionedebole1}
\io \frac{b(x)}{(1+u)^p}{\nabla u}\cdot\nabla \varphi+B \io \frac{|\nabla u|^{2}}{u^{\theta}}\varphi=\io f\varphi\,,
\end{equation}
for every $\varphi \in H^1_0(\Omega)\cap L^{\infty}(\Omega)$. 
\end{theorem}
In the case where $\displaystyle m < \frac{2N}{2N-\theta (N - 2)}=\left(\frac{2^*}{\theta}\right)'$, we are able to prove the existence of an infinite energy solution, belonging to 
$W^{1, \sigma}_0(\Omega)$, with $\displaystyle \sigma=\frac{mN(2-\theta)}{N-\theta m}$ (smaller than 2).
\begin{theorem}\label{teorema_energia_infinita}
Let $0<\theta<1$.
Assume that $f$ is a positive function belonging to $L^m(\Omega)$, with $\displaystyle \frac{N}{2N-\theta (N - 1)}< m < \frac{2N}{2N-\theta (N - 2)}$.
Then there exists a function $u \in W^{1, \sigma}_0(\Omega)$,  strictly positive on $\Omega$, 
such that $\displaystyle \frac{|\nabla u|^2}{u^{\theta}} \in L^1(\Omega)$ and
\begin{equation}\label{formulazionedebole}
\io \frac{b(x)}{(1+u)^p}{\nabla u}\cdot\nabla \varphi+B\io \frac{|\nabla u|^{2}}{u^{\theta}}\varphi=\io f\varphi\,,
\end{equation}
for every $\varphi \in C^1_0(\Omega)$. 
\end{theorem}
In the case where $1\leq \theta<2$,
we are able to prove the same results as in the case $0<\theta<1$,
under a  stronger hypothesis on $f$.
\begin{theorem}\label{teorema_energia_finita2}
Let $1\leq \theta<2$ and $p>\theta-1$.
Assume that $f\in L^m(\Omega)$, with $\displaystyle m \geq \frac{2N}{2N-\theta (N - 2)}$, and satisfies
$$
\mbox{ess inf}\,\{f(x): x \in \omega\}>0 
$$ 
for every $\omega\subset \subset \Omega$.
Then there exists a function $u \in H^1_0(\Omega)$, strictly positive on $\Omega$,
 such that
 $\displaystyle \frac{|\nabla u|^2}{u^{\theta}} \in L^1(\Omega)$ and
$$
\io \frac{b(x)}{(1+u)^p}{\nabla u}\cdot\nabla \varphi+B \io \frac{|\nabla u|^{2}}{u^{\theta}}\varphi=\io f\varphi
$$
for every $\varphi \in H^1_0(\Omega)\cap L^{\infty}(\Omega)$. 
\end{theorem}

\begin{theorem}\label{teorema_energia_infinita2}
Let $1\leq \theta<2$ and $p>\theta-1$.
Assume that $f\in L^m(\Omega)$, with $\displaystyle \frac{N}{2N-\theta (N - 1)}< m < \frac{2N}{2N-\theta (N - 2)}$, and satisfies
$$
\mbox{ess inf}\,\{f(x): x \in \omega\}>0 
$$ 
for every $\omega\subset \subset \Omega$.
Then there exists a function $u \in W^{1, \sigma}_0(\Omega)$, strictly positive on $\Omega$, 
such that $\displaystyle \frac{|\nabla u|^2}{u^{\theta}} \in L^1(\Omega)$ and
$$
\io \frac{b(x)}{(1+u)^p}{\nabla u}\cdot\nabla \varphi+B\io \frac{|\nabla u|^{2}}{u^{\theta}}\varphi=\io f\varphi
$$
for every $\varphi \in C^1_0(\Omega)$. 
\end{theorem}
We remark that if $\displaystyle \theta<\frac{N}{N-1}$ we are able
to prove the existence of solutions when the source $f$ belongs to $L^1(\Omega)$.

We would like to point out the regularizing effects of the lower order term, 
in the case where $p>\theta-1$ and $0<\theta<2$. 
Our results furnish  $H^1_0(\Omega)$ solutions for less summable sources than for 
problem (\ref{without}), since $\displaystyle \frac{2N}{2N-\theta N+2\theta}<\frac{2N}{N(1-p)+2(p + 1)}$.
Even in the case where the source $f$ is less summable, 
we get a better regularity of solutions than for problem (\ref{without}): 
indeed $\displaystyle \sigma=\frac{mN(2-\theta)}{N-\theta m}\geq \frac{Nm(1-p)}{N-m(1+p)}$, as $\displaystyle m\leq \frac N2$ and $p\leq\theta-1$.

In the case where $0<p\leq \theta-1$, we are able to prove the existence of a solution
to problem (\ref{pb}) with the same regularity as the solutions of problem (\ref{without}).
\begin{theorem}\label{teoremastimebdo}
Let $1\leq \theta <2$ and $0<p\leq \theta-1$. 
Assume that $f \in L^m(\Omega)$ and satisfies
$$
\mbox{ess inf}\,\{f(x): x \in \omega\}>0 
$$ 
for every $\omega\subset \subset \Omega$.
\begin{enumerate}
\item
If $\displaystyle m>\frac{N}{2}$, then there exists a strictly positive $H^1_0(\Omega)\cap L^{\infty}(\Omega)$ solution to problem (\ref{pb}). 
\item
If $\displaystyle \frac{2N}{N+2-p(N-2)}\leq m<\frac{N}{2}$, 
then there exists a strictly positive $H^1_0(\Omega)\cap L^{r}(\Omega)$ solution  to problem (\ref{pb}), where
$\displaystyle r=\frac{Nm(1-p)}{N-2m}$.
\item
If $\displaystyle \frac{N}{N+1-p(N-1)}<m<\frac{2N}{N+2-p(N-2)}$, 
then there exists a strictly positive $W^{1,s}_0(\Omega)$ solution  to problem (\ref{pb}), where $\displaystyle s=\frac{Nm(1-p)}{N-m(1+p)}$.
\end{enumerate}
Moreover $\displaystyle \frac{|\nabla u|^2}{u^{\theta}} \in L^1(\Omega)$.
\end{theorem}
In the case where $\theta \geq 2$, the situations changes. Indeed we will prove a non-existence result of finite energy solutions.
Let $\lambda_1(f)$ 
denote the first positive eigenvalue of  
$$
\begin{cases}
\displaystyle
-\Delta u  = \lambda f u & \mbox{in $\Omega$,} \\
\hfill u = 0 \hfill & \mbox{on $\partial\Omega$,}
\end{cases}
$$
where $f$ in $L^q(\Omega)$, with $q>\frac N2$.
Using a result of \cite{luigietaltri}, it is quite easy to prove the following 
\begin{theorem}\label{thmnonexistence}
Let $f\geq 0$, $f\not\equiv 0$, be a $L^q(\Omega)$ function, with $q>\frac N2$.
If either $\theta >2$, or $\theta =2$ and $\lambda_1(f)>\frac{\beta}{B\alpha}$, then there is no
$H^1_0(\Omega)$ solution to problem (\ref{pb}). 
\end{theorem}  


\section{A priori estimates}\label{aprioriestimates}
To prove the existence of solutions to problem (\ref{pb}) we  use the following approximating problems: 
$$
\begin{cases}
\displaystyle
-\div\left(\frac{b(x)}{(1+|T_n(u_n)|)^p}{\nabla u_n}\right) + B \frac{u_n|\nabla u_n|^2}{(|u_n|+\frac 1n)^{\theta+1}} = T_n(f) & \mbox{in $\Omega$,}\\
\hfill u_n = 0 \hfill & \mbox{on $\partial\Omega$,}
\end{cases}
$$
where, for $n \in \N$ and $s\in \R$ 
$$
T_n(s)=\max\{-n,\min\{n,s\}\}\,.
$$ 
These problems are well-posed due to the following result proved in  \cite{bensoussan, bmp, bpm_siam}. 
\begin{theorem}\label{thm_bpm}
Let $f$ be a bounded function. Let $M:\Omega\times \R\to \R^{N^2}$ be a Carath\'edory function 
such that there exist two positive constants $\alpha_0$ an $\beta_0$ such that 
$$
M(x,s)\xi\cdot \xi\geq \alpha_0 |\xi|^2\,,\,\,\,\,\,\,|M(x,s)|\leq \beta_0
$$
for a.e. $x \in \Omega$, for every $(s,\xi) \in \R\times\R^N$.
Let
$g(s)$ be  a Carath\'eodory function such that $g(s)s\geq 0$, $|g(s)|\leq \gamma(s)$, where $\gamma$ is a  continuous, non-negative and increasing function. Then
there exists a $H^{1}_0(\Omega)$ bounded solution to
$$
\begin{cases}
\displaystyle
-\div(M(x,u)\nabla u) + g(u)|\nabla u|^2= f & \mbox{in $\Omega$,} \\
\hfill u = 0 \hfill & \mbox{on $\partial\Omega$\,.}
\end{cases}
$$
\end{theorem}
By Theorem \ref{thm_bpm} the solutions $u_n$ of the above approximating problems  are bounded $H^1_0(\Omega)$ non-negative functions, 
since $f$ is assumed 
to be positive and the lower order term  has the same sign as $u_n$. This implies that $u_n$ 
satisfies
\begin{equation}\label{pa}
\begin{cases}
\displaystyle
-\div\left(\frac{b(x)}{(1+T_n(u_n))^p}{\nabla u_n}\right) + B \frac{u_n|\nabla u_n|^2}{(u_n+\frac 1n)^{\theta+1}} = T_n(f) & \mbox{in $\Omega$,}\\
\hfill u_n = 0 \hfill & \mbox{on $\partial\Omega$.}
\end{cases}
\end{equation}
We are now going to prove some a priori estimates.
The next lemma gives  a control of the lower order term.
\begin{lemma}
Let $u_n$ be the solutions to problems (\ref{pa}). Then it results
\begin{equation}\label{primastima}
B\io \frac{u_n|\nabla u_n|^2}{(u_n+\frac 1n)^{\theta+1}}
\leq
\io f\,.
\end{equation}
\end{lemma}
\begin{proof}
Let us consider $\displaystyle \frac{T_h(u_n)}{h}$, $h>0$, as a test function in (\ref{pa}). We have, dropping the non-negative operator term,
$$
B\io \frac{|\nabla u_n|^2 u_n}{(u_n+\frac 1n)^{\theta +1}}\frac{T_h(u_n)}{h}
\leq
\io f\frac{T_h(u_n)}{h}\,.
$$
It is now sufficient to pass to the limit as $h\to 0$, using Fatou's lemma and the fact that
$\displaystyle \frac{T_h(u_n)}{h}\to 1$ as $h\to 0$.
\end{proof}
We  prove now two a priori estimates on  $u_n$, which
 are true for every $p>0$ and $\theta \in (0,2)$. 
In the sequel $C$ will denote a positive constant independent of $n$; $\mu(E)$ will be the Lebesgue measure of a set $E\subset \R^N$.
\begin{lemma}\label{primolemmaaprioriestimates}
Let $0< \theta <2$. Let 
$f$ be a positive function belonging to $L^m(\Omega)$, with $\displaystyle m \geq \frac{2N}{2N-\theta (N - 2)}$.
Then the solutions $u_n$ to problems (\ref{pa}) are uniformly bounded in $H^1_0(\Omega)$.
Thus there exists a function $u \in H^1_0(\Omega)$ such that,
up to a subsequence, $u_n\to u$ weakly in $H^1_0(\Omega)$ and a.e. in $\Omega$. 
\end{lemma}
\begin{proof}
The assertion follows by proving that
the solutions $u_n$ to problems (\ref{pa}) are uniformly bounded in $H^1_0(\Omega)$.
If we take $(u_n+1)^{\theta}-1$ as a test function in problem (\ref{pa})
we obtain
$$
B\io \frac{|\nabla u_n|^2}{(u_n+\frac 1n)^{\theta+1}}u_n (u_n+1)^{\theta}\leq
B\io \frac{|\nabla u_n|^2}{(u_n+\frac 1n)^{\theta+1}}u_n+ \io fu_n^{\theta}+C\,,
$$ 
dropping the positive operator term.
We can estimate the right hand side using (\ref{primastima}) in order to get
$$
B\io \frac{|\nabla u_n|^2}{(u_n+\frac 1n)^{\theta+1}}u_n (u_n+1)^{\theta}\leq
\io fu_n^{\theta}+C\,.
$$
By working in $\{u_n\geq  1\}$,
the previous inequality gives
$$
\frac B2\int\limits_{\{u_n\geq  1\}} {|\nabla u_n|^2}\leq
\io fu_n^{\theta}+C\leq \int\limits_{\{u_n\geq 1\}} fu_n^{\theta}+C\leq C\int\limits_{\{u_n\geq 1\}} f(u_n-1)^{\theta}+C\,.
$$
We use the Sobolev inequality in the left hand side and the H\"older inequality with exponent $\displaystyle \frac{2^*}{\theta}$ in the last term, 
recalling  that $f$ belongs to $L^{m}(\Omega)$ with $\displaystyle m \geq \frac{2N}{2N-\theta (N - 2)}=\left(\frac{2^{*}}{\theta}\right)'$. Thus
\begin{equation}\label{michaela}
\mathcal{S}\frac B2\left[\int\limits_{\{u_n\geq 1\}} {(u_n-1)^{2^*}}\right]^{\frac{2}{2^*}}
\leq
\frac B2\int\limits_{\{u_n\geq 1\}} {|\nabla u_n|^2}
\leq C
\left[\int\limits_{\{u_n\geq 1\}} (u_n-1)^{2^*}\right]^{\frac{\theta}{2^*}}
+C\,.
\end{equation}
Since we are assuming $\theta<2$, we deduce  that
$$
\displaystyle \int\limits_{\{u_n\geq 1\}} (u_n-1)^{2^*}\leq C\,.
$$ 
It follows from (\ref{michaela}) that
\begin{equation}\label{stima_un_grande}
\int\limits_{\{u_n\geq 1\}}|\nabla u_n|^{2}\leq {C}\,.
\end{equation}
Let us  search for the same kind of estimate in $\{u_n<1\}$.
Taking $T_{1}(u_n)$ as a test function in problem (\ref{pa}), we get
\begin{equation}\label{stima_un_piccolo}
\frac{\alpha}{2^{p}}\int\limits_{\{u_n<1\}}{|\nabla T_{1}(u_n)|^2}
\leq
\alpha\int\limits_{\{u_n< 1\}}\frac{|\nabla T_{1}(u_n)|^2}{(1+u_n)^{p}}
\leq
\io fT_{1}(u_n)\leq \io f
\end{equation}
using  hypothesis (\ref{hyp_a}) and dropping the non-negative lower order term.
As a consequence of estimates (\ref{stima_un_grande}) and (\ref{stima_un_piccolo}), $u_n$ is uniformly bounded in $H^{1}_0(\Omega)$.
By compactness,
there exists a function $u \in H^1_0(\Omega)$ such that,
up to a subsequence, $u_n\to u$ weakly in $H^1_0(\Omega)$ and a.e. in $\Omega$.
\end{proof}

\begin{lemma}\label{secondolemmaaprioriestimates}
Let $0<\theta <2$.
Let $f$ be a positive function belonging to $L^m(\Omega)$, with $\displaystyle \frac{N}{2N-\theta (N - 1)}< m < \frac{2N}{2N-\theta (N - 2)}$.
Then the solutions $u_n$ to problems (\ref{pa}) are uniformly bounded in 
$\displaystyle W^{1,\sigma}_0(\Omega), \sigma=\frac{mN(2-\theta)}{N-\theta m}$.
Thus there exists a function $u \in W^{1,\sigma}_0(\Omega)$ such that,
up to a subsequence, $u_n\to u$ weakly in $W^{1,\sigma}_0(\Omega)$ and a.e. in $\Omega$.
\end{lemma}
\begin{proof}
The assertion follows by proving that the solutions  $u_n$ to problems (\ref{pa}) 
are uniformly bounded in $W^{1,\sigma}_0(\Omega)$.
Take $(u_n+1)^{\theta+2\gamma}-1$, with $\displaystyle \gamma = \frac{2^{*}-\theta m'}{2m' - 2^{*}}$,  as a test function in problems (\ref{pa}).
Note that 
$\gamma<0$: indeed $2^*-\theta m'<0$ and $2m'-2^*> 0$, since $\displaystyle m< \frac N2$. Moreover, $\displaystyle \theta+2\gamma=\frac{2^*(2-\theta)}{2m'-2^*}>0$, as $\theta<2$.
Dropping the non-negative operator term and using estimate 
(\ref{primastima}), we get
$$
B\io \frac{|\nabla u_n|^2}{(u_n+\frac 1n)^{\theta+1}}u_n(u_n+1)^{2\gamma+\theta}\leq
\io f(u_n+1)^{2\gamma+\theta}+C\,.
$$
By working in $\{u_n\geq 1\}$ the previous inequality gives
\begin{equation}\label{primasecondolemma}
\begin{array}{c}
\displaystyle\frac {B}{2(\gamma+1)^2}\int\limits_{\{u_n\geq 1\}} \left|\nabla\left[(u_n+1)^{\gamma+1}-2^{\gamma+1}\right]\right|^2
\leq
\frac {B}{2}\int\limits_{\{u_n\geq 1\}} |\nabla u_n|^2 (u_n+1)^{2\gamma}
\\
\displaystyle
\leq
\int\limits_{\{u_n\geq 1\}} f(u_n+1)^{2\gamma+\theta}+
\int\limits_{\{u_n\leq 1\}} f(u_n+1)^{2\gamma+\theta}
+C
\leq
\int\limits_{\{u_n\geq 1\}} f(u_n+1)^{2\gamma+\theta}+
C
\,.
\end{array}
\end{equation}
The H\"older inequality on the right hand side and the Sobolev inequality on the left one
imply
\begin{equation}\label{prima_stima_energia_infinita}
\begin{array}{c}
\displaystyle \mathcal{S}\left[\int\limits_{\{u_n\geq 1\}} {[(u_n+1)^{\gamma+1}-2^{\gamma+1}]^{2^*}}\right]^{\frac{2}{2^*}}
\leq C \int\limits_{\{u_n\geq 1\}} |\nabla u_n|^2 (u_n+1)^{2\gamma}
\\
\displaystyle\leq C+
C \left[\int\limits_{\{u_n\geq 1\}} (u_n+1)^{(2\gamma+\theta)m'}\right]^{\frac{1}{m'}}\,.
\end{array}
\end{equation}
We remark that the choice of $\gamma$ is equivalent to require $(\gamma+1)2^*=(2\gamma+\theta)m'$; 
moreover $\displaystyle \frac{2}{2^*}\geq \frac{1}{m'}$, due to the hypotheses on $m$ and $\theta$. 
Hence
\begin{equation}\label{prima_stima_energia_infinita_bis}
\int\limits_{\{u_n\geq 1\}} (u_n+1)^{(\gamma+1)2^*}
=\int\limits_{\{u_n\geq 1\}} (u_n+1)^{(2\gamma+\theta)m'}\leq C\,\,\,\,\,\,\,\forall\,n\in \mathbb{N}\,.
\end{equation}
Now,  with $\displaystyle \sigma=\frac{mN(2-\theta)}{N-\theta m}$ as in the statement, and recalling that $\gamma < 0$, let us write
$$
\int\limits_{\{u_n\geq 1\}} |\nabla u_n|^{\sigma}=
\int\limits_{\{u_n\geq 1\}} \frac{|\nabla u_n|^{\sigma}}{(u_n+1)^{-\gamma \sigma}}(u_n+1)^{-\gamma \sigma}\,.
$$
The H\"older inequality  with exponent $\displaystyle \frac{2}{\sigma}$ and estimates (\ref{prima_stima_energia_infinita}) and (\ref{prima_stima_energia_infinita_bis})
give
\begin{equation}\label{stima2ungrande}
\int\limits_{\{u_n\geq 1\}} |\nabla u_n|^{\sigma}
\leq 
\left[
\int\limits_{\{u_n\geq 1\}} \frac{|\nabla u_n|^{2}}{(u_n+1)^{-2\gamma}}
\right]^{\frac{\sigma}{2}}
\left[\int\limits_{\{u_n\geq 1\}}(u_n+1)^{-\gamma\sigma\frac{2}{2-\sigma}}\right]^{\frac{2-\sigma}{2}}\leq C
\end{equation}
since
$\displaystyle -\gamma\frac{2\sigma}{2-\sigma}=(\gamma+1)2^*$. 
It remains to analyse the behaviour of  $\nabla u_n$ on  $\{u_n\leq 1\}$.
Taking $T_1(u_n)$ as a test function in (\ref{pa}) and dropping the non-negative the lower order term, we get
$$
\frac{\alpha}{2^p}\int\limits_{\{u_n\leq 1\}}{|\nabla T_{1}(u_n)|^2}
\leq
\alpha\int\limits_{\{u_n\leq 1\}}\frac{|\nabla T_{1}(u_n)|^2}{(1+u_n)^{p}}
\leq
\io fT_{1}(u_n)\leq \io f
$$
by hypothesis (\ref{hyp_a}).
This last estimate and (\ref{stima2ungrande}) 
imply that $u_n$ is uniformly bounded in $W^{1,\sigma}_0(\Omega)$.
Since $\sigma>1$, 
there exists a function $u \in W^{1,\sigma}_0(\Omega)$ such that,
up to a subsequence, $u_n\to u$ weakly in $W^{1,\sigma}_0(\Omega)$ and a.e. in $\Omega$.
\end{proof}

In the following lemma, we will assume some hypotheses on $p$. This will give, in some cases, 
some better estimates than Lemmata \ref{primolemmaaprioriestimates} and \ref{secondolemmaaprioriestimates}.
\begin{lemma}\label{lemmastessobdo}
Let $0<p<1$. Let $f \in L^m(\Omega)$, $\displaystyle r=\frac{Nm(1-p)}{N-2m}$ and $\displaystyle s=\frac{Nm(1-p)}{N-m(1+p)}$.
\begin{enumerate}
\item
If $\displaystyle m>\frac{N}{2}$,  the solutions of (\ref{pa}) are uniformly bounded in $H^1_0(\Omega) \cap L^{\infty}(\Omega)$.
Thus there exists a function $u \in H^1_0(\Omega)\cap L^{\infty}(\Omega)$ such that,
up to a subsequence, $u_n\to u$ weakly in $H^1_0(\Omega)$ and a.e. in $\Omega$.

\item
If $\displaystyle \frac{2N}{N+2-p(N-2)}\leq m<\frac{N}{2}$,  the solutions of (\ref{pa}) 
are uniformly bounded in $H^1_0(\Omega) \cap L^r(\Omega).$
Thus there exists a function $u \in H^1_0(\Omega)\cap L^r(\Omega)$ such that,
up to a subsequence, $u_n\to u$ weakly in $H^1_0(\Omega)$ and a.e. in $\Omega$.
\item
If $\displaystyle \frac{N}{N+1-p(N-1)}<m<\frac{2N}{N+2-p(N-2)}$,  the solutions of (\ref{pa}) 
are uniformly bounded in $W^{1,s}_0(\Omega)$.
Thus there exists a function $u \in W^{1,s}_0(\Omega)$ such that,
up to a subsequence, $u_n\to u$ weakly in $W^{1,s}_0(\Omega)$ and a.e. in $\Omega$.
\end{enumerate}
\end{lemma}
\begin{proof}
In problems (\ref{pa}) consider as a test function the same test functions as in \cite{bdo}.
With this choice, the lower order term is non-negative and we can take into account only the term given by the operator.
Therefore one can follow the same proofs as in \cite{bdo} to get the above estimates.
\end{proof}

\begin{remark}
Let $p>\theta-1$. Lemmata \ref{primolemmaaprioriestimates} and \ref{secondolemmaaprioriestimates} 
give a further uniform estimate
on $u_n$ than
 Lemma
\ref{lemmastessobdo}. 
 Indeed, if one  chooses $u_n$ as a test function in (\ref{pa}), then, by hypothesis (\ref{hyp_a})
$$
 \io |\nabla u_n|^2\left[\frac{\alpha}{(1+|u_n|)^p}+\frac{B u_n^2}{(u_n+\frac 1n)^{\theta +1}}\right]\leq\io fu_n\,.
$$
If $p> \theta -1$, the lower order term has a leading role in the left hand side of the previous inequality.

\end{remark}

We are  going  to prove the a.e. convergence of the gradients of $u_n$.
We will follow the same technique as in \cite{boccardo_per_puel}. Remark that a similar technique was used for elliptic degenerate problems in \cite{abfot}.
\begin{lemma}\label{lemma_convergenza_gradienti}
Let $u_n$ be the solutions to problems (\ref{pa})
and $u$ be the function found in Lemmata
\ref{primolemmaaprioriestimates}, or \ref{secondolemmaaprioriestimates} 
or
\ref{lemmastessobdo}, according to the summability of $f$.
Up to a subsequence,
$\nabla u_n$ converges to $\nabla u$ a.e. in $\Omega$. 
\end{lemma}
\begin{proof}
Let $h,k>0$. In the sequel $C$ will denote a constant independent of $n, h, k$.
Let us consider  $T_h(u_n-T_k(u))$ as a test function in problems (\ref{pa}).
Then
$$
\io \frac{b(x)}{(1+T_n(u_n))^p}{\nabla u_n\cdot \nabla T_h(u_n-T_k(u))}
\leq
h\io f+
B\io \frac{|\nabla u_n|^2 u_n}{(u_n +\frac 1n)^{\theta +1}}h\,.
$$
By estimate
(\ref{primastima}) on the right hand side and by hypothesis (\ref{hyp_a})  on the left one, 
we get
$$
\io \frac{\nabla u_n\cdot \nabla T_h(u_n-T_k(u))}{(1+T_n(u_n))^{p}}
\leq
Ch\,.
$$
Then we can write
$$
\int\limits_{\{|u_n-T_k(u))|\leq h\}} \frac{|\nabla (u_n-T_k(u))|^2}{(1+u_n)^{p}}
\leq
\io \frac{\nabla (u_n-T_k(u))\cdot\nabla T_h(u_n-T_k(u))}{(1+T_n(u_n))^{p}}
$$
$$
\leq 
Ch-\io \frac{\nabla T_k(u)\cdot\nabla T_h(u_n-T_k(u))}{(1+T_n(u_n))^{p}}\,.
$$
At the limit as $n\to \infty$ one has
$$
\limsup_{n\to \infty}  \int\limits_{\{|u_n-T_k(u)|\leq h\}}\frac{|\nabla T_h(u_n-T_k(u))|^2}{(1+u_n)^p}
\leq Ch\,.
$$
Since $u_n\leq h+k$ in $\{|u_n-T_k(u)|\leq h\}$, we get
\begin{equation}\label{perconvergenzagradienti}
\limsup_{n\to \infty}  \int\limits_{\{|u_n-T_k(u)|\leq h\}}{|\nabla T_h(u_n-T_k(u))|^2}
\leq Ch(1+h+k)^p\,.
\end{equation}
We recall that
$u_n$ is uniformly bounded in $W^{1,\eta}_0(\Omega)$, where $\eta$ equals 2 or $\sigma$ or $s$, according 
to the statements of 
Lemmata \ref{primolemmaaprioriestimates},
\ref{secondolemmaaprioriestimates} 
and
\ref{lemmastessobdo}.
Let $q\in (1,\eta)$. 
We can write
$$
\io|\nabla (u_n-u)|^q=
$$
$$
\int\limits_{\{|u_n-u|\leq h,|u|\leq k\}}|\nabla (u_n-u)|^q+
\int\limits_{\{|u_n-u|\leq h,|u|> k\}}|\nabla (u_n-u)|^q+
\int\limits_{\{|u_n-u|> h\}}|\nabla (u_n-u)|^q\,.
$$
Using the H\"older inequality with exponent $\frac 2q$ 
on the first term of the right hand side and exponent $\frac{\eta}{q}$ on the other ones, 
we have 
$$
\io|\nabla (u_n-u)|^q
\leq
$$
$$
\leq
C\left[\int\limits_{\{|u_n-u|\leq h,|u|\leq k\}}|\nabla (u_n-u)|^2\right]^{\frac q2}
+
C
\left[\mu({\{|u|> k\}})^{1-\frac{q}{\eta}} + \mu({\{|u_n-u|> h\}})^{1-\frac{q}{\eta}}\right]
\,,
$$
where we have used that $u_n$ is uniformly bounded in $W^{1,\eta}_0(\Omega)$ to estimate the last two terms. 
By (\ref{perconvergenzagradienti}) the limit as $n\to \infty$ gives
$$
\limsup_{n\to \infty}\io |\nabla (u_n-u)|^q
\leq 
[Ch(1+k+h)^p]^{\frac q2}+
C \mu({\{|u|> k\}})^{1-\frac{q}{\eta}}\,.
$$
The limit as $h\to 0$ implies
$$
\limsup_{n\to \infty}\io |\nabla (u_n-u)|^q
\leq 
C \mu({\{|u|> k\}})^{1-\frac{q}{\eta}}\,.
$$
At the limit as $k\to +\infty$,
$\mu({\{|u|> k\}})$ converges to $0$. Therefore $\nabla u_n\to \nabla u$ in $L^q(\Omega)$. Up to a subsequence, $\nabla u_n\to \nabla u$ a.e. in $\Omega$.
\end{proof}

\section{Existence results in the case $0<\theta<1$}
To prove the existence of solutions
to problem (\ref{pb}),
the key point is to prove that the function $u$ found 
by compactness in the lemmata of Section \ref{aprioriestimates}
is strictly positive.
In the case $0<\theta<1$, we use a  technique similar to that  in \cite{boccardo_per_puel}.
\begin{proposition}\label{funzionepositiva}
Let $0<\theta<1$. Let $u_n$ and
$u$ be
as in Lemma
\ref{lemma_convergenza_gradienti}. Then $u>0$.
\end{proposition}
\begin{proof}
We define, for $s\geq 0$, 
$$
{H_n(s)}= \int_0^s\frac{t (1+T_n(t))^p}{\alpha(t+\frac 1n)^{\theta+1}}dt
\,,\,\,\,\,\,\,{H(s)}= \int_0^s\frac{(1+t)^p}{\alpha t^{\theta}}dt\,.
$$
Observe that $H$ is well-defined, since $\theta<1$.
We choose  $e^{-BH_n(u_n)}\phi$, where $\phi$ is a positive $C^{\infty}_0(\Omega)$ function, as a test function in
 (\ref{pa}). This gives
$$
\io \frac{b(x)}{(1+T_n(u_n))^p}{\nabla u_n\cdot \nabla \phi\, e^{-BH_n(u_n)}}-\io T_n(f) e^{-BH_n(u_n)}\phi=
$$
$$
=B\io \frac{b(x)}{(1+T_n(u_n))^p}{e^{-BH_n(u_n)} |\nabla u_n|^2}\phi H_n'(u_n)-
B\io\frac{e^{-BH_n(u_n)} |\nabla u_n|^2 u_n}{(\frac 1n+u_n)^{\theta + 1}}\phi
$$
$$
\geq B\io \frac{\alpha}{(1+T_n(u_n))^p}{e^{-BH_n(u_n)} |\nabla u_n|^2}\phi H_n'(u_n)-
B\io\frac{e^{-BH_n(u_n)} |\nabla u_n|^2 u_n}{(\frac 1n+u_n)^{\theta + 1}}\phi
$$
by hypothesis (\ref{hyp_a}).
The last quantity is positive,
due to the choice of $H_n$ and $\phi$. As a consequence
$$
\io \frac{b(x)}{(1+T_n(u_n))^p}{\nabla u_n \cdot \nabla \phi\, e^{-BH_n(u_n)}}\geq \io T_n(f) e^{-BH_n(u_n)}\phi\geq \io T_1(f) e^{-BH_n(u_n)}\phi\,.
$$
Now, we set 
$$
P_n(s)=\int_0^s \frac{e^{-BH_n(t)}}{(1+T_n(t))^p}dt\,,\,\,\,\,\,\,\,P(s)=\int_0^s \frac{e^{-BH(t)}}{(1+t)^p}dt\,.
$$ 
With these definitions, we remark that we have just proved that the inequality
$$
-\div (b(x)\nabla (P_n(u_n)))\geq T_1(f)e^{-BH_n(u_n)}
$$
holds distributionally.
Observe that for every $n\in \N$, $P_n(u_n)\in H^1_0(\Omega)$, since $P_n'$ is bounded and $u_n\in H^1_0(\Omega)$.
Let $z_n$ be the $H^{1}_0(\Omega)$ solution to $$-\div(b(x)\nabla z_n)=T_1(f)e^{-BH_n(u_n)}\,;$$ 
let $z$ be the $H^{1}_0(\Omega)$ solution to $$-\div(b(x)\nabla z)=T_1(f)e^{-BH(u)}\,.$$
Then
$$
-\div (b(x)\nabla (P_n(u_n)))\geq -\div (b(x)\nabla z_n)\,.
$$
The comparison principle in $H^1_0(\Omega)$ implies that $P_n(u_n(x))\geq z_n(x)$ for a.e. $x\in \Omega$.
Up to a subsequence,
$z_n\to z$ weakly in $H^1_0(\Omega)$
and  a.e. in $\Omega$.
At the limit a.e. in $\Omega$, as $n\to +\infty$, we have $P(u)\geq z$.
By the strong maximum principle  $z>0$ and so $P(u)>0$.
Since $P$ is strictly increasing, $u>0$ in $\Omega$.
\end{proof}
\begin{corollary}\label{primo_corollario}
Let $0<\theta<1$. Let $u_n$ and
$u$ be
as in Lemma
\ref{lemma_convergenza_gradienti}. Then $\displaystyle \frac{|\nabla u|^2}{u^{\theta}} \in L^1(\Omega)$.
\end{corollary}
\begin{proof}
We pass to the limit  in  (\ref{primastima}).
The a.e. convergence of $u_n$ to $u$ (see Lemmata \ref{primolemmaaprioriestimates},
\ref{secondolemmaaprioriestimates} 
and
\ref{lemmastessobdo}), the a.e. convergence of $\nabla u_n$ to $\nabla u$ (see Lemma \ref{lemma_convergenza_gradienti})
and Proposition \ref{funzionepositiva}
 imply
$$
B\io \frac{|\nabla u|^2}{u^\theta}\leq \io f
$$
by Fatou's lemma.
\end{proof}
We are going to prove Theorem \ref{teorema_energia_finita}. 
\begin{proof}
We are going to prove that the function
$u$ found in Lemma \ref{primolemmaaprioriestimates}, and studied in Lemma \ref{lemma_convergenza_gradienti},
Proposition \ref{funzionepositiva} and Corollary \ref{primo_corollario},  is a weak solution to problem (\ref{pb}). We use the same technique as in \cite{boccardo_per_puel}.

We will prove that (\ref{formulazionedebole1}) holds true for every  positive and bounded $\varphi\in H^1_0(\Omega)$. 
The general case follows from the fact that  every such function
$\varphi$ can  be written as $\varphi_+-\varphi_-$ with 
$\varphi_{\pm}$ bounded, positive and belonging to $H^1_0(\Omega)$.

We pass to the limit as $n\to \infty$ in 
$$
\io \frac{b(x)}{(1+T_n(u_n))^p}\nabla u_n\cdot\nabla \varphi
+B\io \frac{|\nabla u_n|^2 u_n}{(u_n+\frac 1n)^{1+\theta}}\varphi=
\io T_n(f)\varphi\,,
$$
where $\varphi$ is a positive bounded $H^1_0(\Omega)$ function.
Regarding the first term we observe that
$\displaystyle \frac{b(x)}{(1+T_n(u_n))^p}{\nabla \varphi}$ strongly converges to $\displaystyle \frac{b(x)}{(1+u)^p}{\nabla \varphi}$ 
in $L^2(\Omega)$ 
and
$\nabla u_n$ weakly converges to $\nabla u$  in $L^2(\Omega)$.
For the second one we use the a.e. convergence  of $\nabla u_n$, proved in Lemma \ref{lemma_convergenza_gradienti}. 
Fatou's lemma
implies
\begin{equation}\label{first-inequality-first-theorem}
\io \frac{b(x)}{(1+u)^p}\nabla u\cdot\nabla \varphi
+B\io \frac{|\nabla u|^2}{u^{\theta}}\varphi\leq \io f\varphi\,.
\end{equation}
The proof of  the opposite inequality is more delicate.
To this aim, we define, for $n \in \N$ and $s\geq 0$, 
$$
H_{\frac 1n}(t)=\int_0^t\frac{B(1+s)^p}{\alpha(s+\frac 1n)^{\theta}}ds\,,\,\,\,\,\,\,\,\,
H_{0}(t)=\int_0^t\frac{B(1+s)^p}{\alpha s^{\theta}}ds\,.
$$
$H_0$ is well-posed, since $\theta < 1$.
Let us consider 
$$
v=e^{- H_{\frac 1n}(u_n)}e^{ H_{\frac 1j}(T_j(u))} \varphi\,,
$$
where $j\in \N$ and $\varphi$ is a positive bounded $H^1_0(\Omega)$ function, as a test function in (\ref{pa}). 
Then
$$
\io \frac{b(x)}{(1+T_n(u_n))^p}\nabla u_n\cdot \nabla \varphi\, e^{-H_{\frac 1n}(u_n)}e^{H_{\frac 1j}(T_j(u))}
$$
$$
+\frac{B}{\alpha}
\io \frac{b(x)}{(1+T_n(u_n))^p}{\varphi}\,e^{-H_{\frac 1n}(u_n)}
e^{H_{\frac 1j}(T_j(u))}\frac{\nabla u_n\cdot \nabla T_j(u)}{(T_j(u)+\frac 1j)^{\theta}}
(T_j(u)+1)^{p}
$$
$$
= \io T_n(f) e^{-H_{\frac 1n}(u_n)}e^{ H_{\frac 1j}(T_j(u))} \varphi
+\frac{B}{\alpha}\io \frac{b(x)|\nabla u_n|^2}{(1+T_n(u_n))^{p}}\frac{(1+u_n)^{p}}{(\frac 1n+u_n)^{\theta}}\varphi
\, e^{-H_{\frac 1n}(u_n)}e^{H_{\frac 1j}(T_j(u))}
$$
$$
-B\io 
\frac{|\nabla u_n|^2 u_n}{(\frac 1n+u_n)^{\theta+1}}
e^{-H_{\frac 1n}(u_n)}e^{H_{\frac 1j}(T_j(u))}\varphi\,.
$$
Note that by hypothesis (\ref{hyp_a}) and inequality
$$
\left(\frac{u_n+1}{1+T_n(u_n)}\right)^p \geq 1>\frac{u_n}{u_n+\frac{1}{n}}\,,
$$
 the sum of the  last two terms is non-negative.
At the limit as $n\to \infty$ we have
$$
\io \frac{b(x)}{(1+u)^p}\nabla u\cdot \nabla \varphi\, e^{- H_{0}(u)}e^{ H_{\frac 1j}(T_j(u))}
$$
$$
+\frac{B}{\alpha}
\io \frac{b(x)}{(1+u)^p}\varphi\,e^{-H_{0}(u)}
e^{H_{\frac 1j}(T_j(u))}\frac{\nabla u\cdot \nabla T_j(u)}{(T_j(u)+\frac 1j)^{\theta}}(T_j(u)+1)^{p}
$$
$$
\geq \io f e^{-H_{0}(u)}e^{H_{\frac 1j}(T_j(u))} \varphi
+\frac{B}{\alpha}\io \frac{b(x)|\nabla u|^2}{u^{\theta}}\varphi
\, e^{-H_{0}(u)}e^{H_{\frac 1j}(T_j(u))}
$$
$$
-B\io 
\frac{|\nabla u|^2}{u^{\theta}}
e^{-H_{0}(u)}e^{H_{\frac 1j}(T_j(u))}\varphi
\,,
$$
using the weak convergence of $u_n$ to $u$ in $H^1_0(\Omega)$
in the left  hand side and Fatou's lemma in the right one.
Now we pass to the limit as  $j\to \infty$, using that 
$
e^{- H_{0}(u)}e^{ H_{\frac 1j}(T_j(u))}\leq 1
$
and Corollary \ref{primo_corollario}.
We obtain
\begin{equation}\label{second-inequality-first-theorem}
\io \frac{b(x)}{(1+u)^p}\nabla u\cdot \nabla \varphi
\geq \io f \varphi
-B\io \varphi
\frac{|\nabla u|^2}{u^{\theta}}
\,.
\end{equation}
Inequalities (\ref{first-inequality-first-theorem}) 
and (\ref{second-inequality-first-theorem}) imply that 
$$
\io \frac{b(x)}{(1+u)^p}\nabla u\cdot \nabla \varphi
+B\io \varphi
\frac{|\nabla u|^2}{u^{\theta}}
= \io f \varphi
$$
for every  positive and bounded $\varphi\in H^1_0(\Omega)$.
\end{proof}
We are going to prove Theorem \ref{teorema_energia_infinita}.
\begin{proof}
We are going to prove that the function $u$ found in Lemma \ref{secondolemmaaprioriestimates}
and studied in
Lemma \ref{lemma_convergenza_gradienti}, Proposition \ref{funzionepositiva} and Corollary \ref{primo_corollario},
 is a weak solution to problem (\ref{pb}). We use the same technique as in \cite{bmp, porretta}.

We first prove (\ref{formulazionedebole}) for every positive  $C^1_0(\Omega)$ function $\varphi$. 
With the same argument as in the previous theorem (i.e., using Fatou's lemma) one can prove that 
\begin{equation}\label{first-inequality-second-theorem}
\io \frac{b(x)}{(1+u)^p}\nabla u\cdot\nabla \varphi
+B\io \frac{|\nabla u|^2}{u^{\theta}}\varphi\leq \io f\varphi\,.
\end{equation}
To prove the opposite inequality, we slightly modify the previous proof, since we no longer have uniform estimates of $u_n$ in $H^1_0(\Omega)$.
Observe that, however, $T_k(u_n)$ is uniformly bounded in $H^1_0(\Omega)$. Indeed, it is sufficient to consider
$T_k(u_n)$ as a test function in (\ref{pa}): we obtain
$$
\int\limits_{\{u_n\leq k\}}|\nabla T_k(u_n)|^2\leq Ck(1+k)^p \,\,\,\,\,\,\,\forall\, n\in \N
$$
by hypothesis (\ref{hyp_a}).
 We will  use, for $k \in \N$ and $s\in \R$ 
$$
R_k(s)=\left\{
\begin{array}{ll}
1, & s\leq k
\\
k+1-s, & k\leq s\leq k+1
\\
0, & s>k+1\,,
\end{array}
\right.
$$
 to define a test function.
We set, for $t\geq 0$,
$$
H_{\frac 1n}(t)=\int_0^t\frac{B(1+s)^p}{\alpha(s+\frac 1n)^{\theta}}ds\,,\,\,\,\,\,\,\,\,
H_{0}(t)=\int_0^t\frac{B(1+s)^p}{\alpha s^{\theta}}ds\,.
$$
This is possible, since $\theta < 1$.
We consider 
$$
v=e^{- H_{\frac 1n}(u_n)}e^{ H_{\frac 1j}(T_j(u))} R_k(u_n)\varphi\,,
$$ 
where
 $\varphi$ is a positive $C^1_0(\Omega)$ function and $j\in N$, as a test function 
in (\ref{pa}).
Then
$$
\io \frac{b(x)}{(1+T_n(u_n))^p}\nabla u_n\cdot \nabla \varphi\, 
e^{-H_{\frac 1n}(u_n)}e^{H_{\frac 1j}(T_j(u))}R_k(u_n)
$$
$$
+\frac{B}{\alpha}
\io \frac{b(x)}{(1+T_n(u_n))^p}{\varphi}\,e^{-H_{\frac 1n}(u_n)}
e^{H_{\frac 1j}(T_j(u))}\frac{\nabla u_n\cdot \nabla T_j(u)}{(T_j(u)+\frac 1j)^{\theta}}
(T_j(u)+1)^{p}R_k(u_n)
$$
$$
=\io T_n(f) e^{-H_{\frac 1n}(u_n)}e^{H_{\frac 1j}(T_j(u))} \varphi R_k(u_n)+
\int\limits_{\{k\leq u_n\leq k+1\}} \frac{b(x)|\nabla u_n|^2}{(1+T_n(u_n))^{p}}\varphi
\, e^{-H_{\frac 1n}(u_n)}e^{H_{\frac 1j}(T_j(u))}
$$
$$
+\frac{B}{\alpha}\io \frac{b(x)|\nabla u_n|^2}{(1+T_n(u_n))^{p}}\frac{(1+u_n)^{p}}{(\frac 1n+u_n)^{\theta}}\varphi
\, e^{-H_{\frac 1n}(u_n)}e^{H_{\frac 1j}(T_j(u))}R_k(u_n)
$$
$$
-B\io 
\frac{|\nabla u_n|^2 u_n}{(\frac 1n+u_n)^{\theta+1}}
 \,e^{-H_{\frac 1n}(u_n)}e^{H_{\frac 1j}(T_j(u))}R_k(u_n)\varphi\,.
$$
The sum of the last two terms is positive, since $b(x)\geq \alpha$ by  hypothesis (\ref{hyp_a})
and by inequality
$$
\left(\frac{u_n+1}{1+T_n(u_n)}\right)^p \geq 1\geq \frac{u_n}{u_n+\frac{1}{n}}\,.
$$
Dropping the non-negative term $\displaystyle\int\limits_{\{k\leq u_n\leq k+1\}} \frac{b(x)|\nabla u_n|^2}{(1+T_n(u_n))^{p}}\varphi
\, e^{-H_{\frac 1n}(u_n)}e^{H_{\frac 1j}(T_j(u_j))}$,
at the limit as $n\to \infty$ we have, by Fatou's lemma, the weak convergence of $u_n$ in $W^{1,\sigma}_0(\Omega)$ and the weak convergence of
$T_k(u_n)$ in $H^1_0(\Omega)$,
$$
\io \frac{b(x)}{(1+u)^p}\nabla u\cdot \nabla \varphi\, e^{-H_{0}(u)}e^{H_{\frac 1j}(T_j(u))}R_k(u)+
$$
$$
+\frac{B}{\alpha}
\io \frac{b(x)}{(1+u)^p}{\varphi}\,e^{-H_{0}(u)}
e^{H_{\frac 1j}(T_j(u))}\frac{\nabla u\cdot \nabla T_j(u)}{(T_j(u)+\frac 1j)^{\theta}}
(T_j(u)+1)^{p}R_k(u)
$$
$$
\geq \io f e^{-H_{0}(u)}e^{H_{\frac 1j}(T_j(u))} \varphi\,R_k(u)+
\frac{B}{\alpha}\io \frac{b(x)|\nabla u|^2}{u^{\theta}}\varphi
\, e^{-H_{0}(u)}e^{H_{\frac 1j}(T_j(u))}R_k(u)
$$
$$
-B\io 
\frac{|\nabla u|^2}{u^{\theta}}
\,e^{-H_{0}(u)}e^{H_{\frac 1j}(T_j(u))}R_k(u)\varphi \,.
$$
As in the previous proof, it is now sufficient to pass to the limit as $j\to \infty$ first, using that 
$
e^{-H_{0}(u)}e^{H_{\frac 1j}(T_j(u))}\leq 1
$ and Corollary \ref{primo_corollario}, and then to the limit as $k\to \infty$, using that $R_k(u)$ tends to 1. 
We thus obtain
\begin{equation}\label{second-inequality-second-theorem}
\io \frac{b(x)}{(1+u)^p}\nabla u\cdot\nabla \varphi
\geq \io f\varphi - B\io \frac{|\nabla u|^2}{u^{\theta}}\varphi\,.
\end{equation}
Inequalities (\ref{first-inequality-second-theorem}) and (\ref{second-inequality-second-theorem}) imply that 
\begin{equation}\label{last-inequality-second-theorem}
\io \frac{b(x)}{(1+u)^p}\nabla u\cdot\nabla \varphi
+ B\io \frac{|\nabla u|^2}{u^{\theta}}\varphi
=\io f\varphi 
\end{equation}
for every  positive $\varphi\in C^1_0(\Omega)$.
Now, let $\varphi$ any $C^1_0(\Omega)$ function. We define  $\varphi^{\varepsilon}_{\pm}=\rho^{\varepsilon}*\varphi_{\pm}$ as the convolution of 
 a mollifier $\rho^{\varepsilon}$, for $\varepsilon>0$, with  $\varphi_{\pm}$.
Then $\varphi^{\varepsilon}_{\pm}$ is a positive $C^1_0(\Omega)$ function, for $\varepsilon$ sufficiently small. By (\ref{last-inequality-second-theorem}) we have
$$
\io \frac{b(x)}{(1+u)^p}\nabla u\cdot\nabla (\varphi^{\varepsilon}_{-}-\varphi^{\varepsilon}_{+})
+ B\io \frac{|\nabla u|^2}{u^{\theta}}(\varphi^{\varepsilon}_{-}-\varphi^{\varepsilon}_{-})
=\io f(\varphi^{\varepsilon}_{-}-\varphi^{\varepsilon}_{-})\,. 
$$
Since $\varphi^{\varepsilon}_{-}-\varphi^{\varepsilon}_{-}\to \varphi$ uniformly in $\Omega$ and in $W^{1,q}_0(\Omega)$ for every $q\geq 1$, as $\varepsilon \to 0$, the result follows.
\end{proof}


\section{Existence results in the case  $1\leq \theta<2$}
As in the above case, we need to prove that
the function $u$ found in Section \ref{aprioriestimates} is not $0$ in $\Omega$.
To this aim,
we are going to prove that for every $\omega\subset\subset \Omega$
there exists a positive constant $c_{\omega}$ such that
the solutions $u_n$ to problems (\ref{pa})
satisfy $u_n\geq c_{\omega}$ in $\omega$ for every $n\in \N$. 
  We will follow a similar technique to that one in \cite{luigietaltri}.
The following theorem, 
proved in \cite{pellacci} (and in \cite{luigietaltri}), will be useful to us.

\begin{theorem}\label{K-O}
Let $B: \Omega \times \R\to \R$ be a Carath\'eodory function
such that for every $\omega \subset \subset \Omega$ there exists $m_{\omega}>0$ such that 
$B(x,s)\geq m_{\omega}l(s)$ for a.e. $x\in \Omega$ and for every $s\geq 0$.
Assume that $l:\R^+\to \R^+$ is a continuous increasing function such that $l(s)/s$ 
is increasing for $s$ sufficiently large and  for some $t_0>0$
\begin{equation}\label{cdtko}
\displaystyle \int_{t_0}^{+\infty}\frac{dt}{\sqrt{\int_0^t l(s)ds}}<+\infty\,.
\end{equation}

Then for every $\omega\subset \subset \Omega$ there exists a constant $C_{\omega}>0$ 
such that every sub-solution $v\in H^1_{loc}(\Omega)$ of $-\div(b(x)\nabla v) +B(x,v)=0$ such that
$v^+ \in L^{\infty}_{loc}(\Omega)$ and $B(x,v^+)\in L^1_{loc}(\Omega)$
satisfies 
$v\leq C_{\omega}$ in $\omega$. 
\end{theorem}
\begin{remark}
We recall that a sub-solution of
$-\div(b(x)\nabla v) +l(v)g(x)=0$ is a $W^{1,1}_{loc}(\Omega)$ function such that
$$
\io b(x)\nabla v\cdot \nabla \phi +\io l(v)g(x)\phi\leq 0
$$ 
for every $C^{\infty}_c(\Omega)$ positive function $\phi$.
\end{remark}
\begin{remark}
In the literature condition (\ref{cdtko}) is called the Keller-Osserman condition, due to the papers \cite{keller, osserman}
on semilinear equations. 
\end{remark}

\begin{proposition}\label{proposition_ko}
Let $1\leq \theta<2$. Let $u_n$ be the solutions of (\ref{pa}).
Then for every $\omega\subset\subset \Omega$ there exists a strictly positive constant $c_{\omega}$
such that $u_n\geq c_{\omega}$ in $\omega$ for every $n \in \N$.
\end{proposition}
\begin{proof}
{\it Step 1.}  
Let $u_n$ be a $H^{1}_0(\Omega)\cap L^{\infty}(\Omega)$ solution to (\ref{pa}). 
We perform a change of variable in order to get a sub-solution of an elliptic semi-linear problem, as in Theorem \ref{K-O}.

We set $\displaystyle a_n(s)=\frac{1}{(1+T_n(s))^p}$. Then
$u_n$ satisfies, distributionally,
$$
-\div\left(b(x)a_n(u_n)\nabla u_n\right)+\frac{B}{u_n^{\theta}}{|\nabla u_n|^2}\geq T_1(f)\,,
$$
that is,
\begin{equation}\label{disupropko}
-\div(b(x)\nabla u_n) a_n(u_n)-a_n'(u_n)b(x)|\nabla u_n|^2+\frac{B}{u_n^{\theta}}{|\nabla u_n|^2}\geq T_1(f)\,.
\end{equation}
Let $\displaystyle k_n(t)=\int_1^t \frac{B}{\alpha r^{\theta}a_n(r)}dr$ and
$\displaystyle
\psi_n(s)=\int_s^1 e^{-k_n(t)}a_n^{\frac{\beta}{\alpha}}(t)dt\,.
$
We remark that 
\begin{equation}
\label{psi}
 \psi_n'(s)=-a_n^{\frac{\beta}{\alpha}}(s) e^{-k_n(s)}\,,\,\,\,\,\,\,\,\,\,\,\,\,\frac{\psi_n''(s)}{\psi_n'(s)}=\frac{\beta}{\alpha}\frac{a_n'(s)}{a_n(s)}-\frac{B}{\alpha s^{\theta}a_n(s)}\,.
\end{equation}
We define  $v_n=\psi_n(u_n)$. Then 
$$
\div(b(x) \nabla v_n) =
\div(b(x) \psi_n'(u_n)\nabla u_n)=
\psi_n'(u_n)\div(b(x)\nabla u_n)+b(x){\psi_n''(u_n)}|\nabla u_n|^2
$$
and therefore
$$
-a_n(u_n)\div(b(x)\nabla u_n)=-a_n(u_n)\frac{\div(b(x)\nabla v_n)}{\psi_n'(u_n)} +a_n(u_n)b(x)\frac{\psi_n''(u_n)}{\psi_n'(u_n)}|\nabla u_n|^2\,.
$$
By inequality (\ref{disupropko}) we have
$$
T_1(f)\leq 
\displaystyle -a_n(u_n)\frac{\div(b(x)\nabla v_n)}{\psi_n'(u_n)} +a_n(u_n)b(x)\frac{\psi_n''(u_n)}{\psi_n'(u_n)}|\nabla u_n|^2
-a_n'(u_n)b(x)|\nabla u_n|^2+\frac{B}{u_n^{\theta}}|\nabla u_n|^2\,.
$$
Using that $a_n'(s)\leq 0$, $\displaystyle \frac{\psi_n''(s)}{\psi_n'(s)}\leq 0$ and hypothesis (\ref{hyp_a}) we obtain 
$$
T_1(f)\leq -a_n(u_n)\frac{\div(b(x)\nabla v_n)}{\psi_n'(u_n)} +a_n(u_n)\alpha\frac{\psi_n''(u_n)}{\psi_n'(u_n)}|\nabla u_n|^2
-a_n'(u_n)\beta|\nabla u_n|^2+\frac{B}{u_n^{\theta}}|\nabla u_n|^2\,.
$$
Due to (\ref{psi})
$$
T_1(f)\leq -a_n(u_n)\frac{\div(b(x)\nabla v_n)}{\psi_n'(u_n)}\,.
$$
Observing that $\psi_n'(s)=-a_n^{\frac{\beta}{\alpha}}(s) e^{-k_n(s)}\leq 0$, 
$v_n$ satisfies
$$
0\geq -\div(b(x)\nabla v_n)+{T_1(f)}e^{-k_n(\psi_n^{-1}(v_n))}a_n^{\frac{\beta}{\alpha}-1}(\psi_n^{-1}(v_n)) 
\,.$$
{\it Step 2.}
We now study, for $s\geq 0$ 
$$
s\to e^{-k_n(\psi_n^{-1}(s))}a_n^{\frac{\beta}{\alpha}-1}(\psi_n^{-1}(s))\,.
$$
We remark that
 $\psi_n^{-1}(s)\leq 1$, since   $s\geq 0=\psi_n(1)$ and $\psi_n$ is decreasing.
Therefore
\begin{equation}\label{disu_facile}
a_n^{\frac{\beta}{\alpha}-1}(\psi_n^{-1}(s))\geq a_n^{\frac{\beta}{\alpha}-1}(1)=a_0\,,
\end{equation}
as $a_n$ is decreasing.

Recalling that
$$
\psi_n(s)=\int_s^1 e^{-k_n(t)}a_n^{\frac{\beta}{\alpha}}(t)dt\,,\,\,\,\,\,\,\,\,\,k_n(t)=\int_1^t \frac{B}{\alpha r^{\theta}a_n(r)}dr
$$
and
$$
\left\{
\begin{array}{ll}
a_n(s)=a_1(s)\,, & s\leq 1
\\
a_n(s)\leq a_1(s)\,, & s>1
\end{array}
\right.
$$ 
it is not difficult to prove that
\begin{equation}\label{disu_psi_n}
\psi_n(s)\geq \psi_1(s)\,,
\end{equation}
distinguishing the cases $s\leq 1$ and $s> 1$.
Now, inequality (\ref{disu_psi_n}) and the fact that $\psi_n$ is decreasing imply that $\psi_n^{-1}(s)\leq \psi_1^{-1}(s)$ for every $s\geq 0$.
Recalling that $\psi_n^{-1}(s)\leq 1$ and $a_n(s)=a_1(s)\geq 0$ for $s\geq 1$, we deduce easily that
\begin{equation}\label{disu_difficile}
\displaystyle e^{-k_n(\psi_n^{-1}(s))}
\geq
e^{-k_1(\psi_1^{-1}(s))}\,.
\end{equation}
Due to (\ref{disu_facile}) and (\ref{disu_difficile}), $v_n$ satisfies 
$$
0\geq -\div(b(x)\nabla v_n)+B(x,v_n)
$$
with 
$$
B(x,s)=\left\{
\begin{array}{ll}
{T_1(f)}a_0\,l(s)\,,& s\geq 0
\\
0\,,& s\leq 0\,,
\end{array}
\right.
$$
where $l(s)=e^{-k_1(\psi_1^{-1}(s))}-1$, $s\geq 0$.

{\it Step 3.}
We are going to prove that $l$ satisfies the hypotheses of Theorem \ref{K-O}.

We observe that $l$ is continuous and  increasing, since $\psi_1^{-1}$ is decreasing and $k_1$ is increasing.
We claim that $l(s)/s$ is increasing for $s$ sufficiently large.
This is equivalent to prove that
$
\displaystyle Y(t)=\frac{l(\psi_1(t))}{\psi_1(t)}
$
is decreasing for small positive $t$.
Now,
$Y'(t)<0$ if and only if 
\begin{equation}\label{conditio_difficilissima}
l'(\psi_1(t))\psi_1(t)-\int_1^t l'(\psi_1(s))\psi_1'(s)ds> 0\,.
\end{equation}
We remark that $\displaystyle l'(\psi_1(s))=\frac{B}{\alpha s^{\theta}a_1^{{\frac{\beta}{\alpha}}+1}(s)}$.
Let $w_0 \in (0,1)$ be such that
$\displaystyle h(t)=l'(\psi_1(t))=\frac{B}{\alpha t^{\theta}a_1^{\frac{\beta}{\alpha}+1}(t)}$ is decreasing in $(0,w_0]$.
Therefore
$$
l'(\psi_1(t))\psi_1(t)-\int_1^t l'(\psi_1(s))\psi_1'(s)ds=
\int_{t}^1 e^{-k_1(s)}a_1^{\frac{\beta}{\alpha}}(s)
\left[
h(t)-h(s)
\right]ds
$$
$$
\geq \int_{w_0}^1 e^{-k_1(s)}a_1^{\frac{\beta}{\alpha}}(s)
\left[
h(t)-h(s)
\right]ds
$$
due to the choice of $w_0$.
Let 
$$
M_1=
\int_{w_0}^1 e^{-k_1(s)}
{a_1^{\frac{\beta}{\alpha}}(s)}
ds\,,\,\,\,\,\,\,\,
M_2=
\int_{w_0}^1 
e^{-k_1(s)}a_1^{\frac{\beta}{\alpha}}(s)
h(s)
ds\,.
$$
We have proved that  
$$
l'(\psi_1(t))\psi_1(t)-\int_1^t l'(\psi_1(s))\psi_1'(s)ds\geq M_1 h(t)-M_2\,.
$$
If $t$ is sufficiently small, the last quantity is positive, since $h$ is decreasing for small positive $t$. Therefore (\ref{conditio_difficilissima}) holds.

We are going to study the last condition on $l$, that is, the existence of a positive $t_0$
such that
\begin{equation}\label{condition_integrale}
\int_{t_0}^{+\infty}\frac{dt}{\sqrt{\int_0^t l(s)ds}}<\infty\,.
\end{equation}
Using the change of variable $\tau =\psi_1^{-1}(s)$ we get
$$
\int_0^t l(s)ds=\int_0^t [e^{-k_1(\psi_1^{-1}(s))}-1]ds=
\int_{\psi_1^{-1}(t)}^{1}[e^{-k_1(\tau)}-1]a_1^{\frac{\beta}{\alpha}}(\tau)e^{-k_1(\tau)} d\tau
\,.
$$
It is easy to see that
$
\displaystyle
e^{-k_1(\tau)}-1 
\geq
\frac 12
e^{-k_1(\tau)}
$ for $\tau\leq \tau_0$ sufficiently small. Moreover $a_1(\tau) \geq \frac 12$, for $\tau \leq 1$.
Therefore
it suffices  to find $t_0$ sufficiently large ($t_0>\psi_1(\tau_0)$) such that
$$
\int_{t_0}^{+\infty}\frac{dt}{\sqrt{\int_{\psi_1^{-1}(t)}^{1}e^{-2k_1(\tau)} d\tau}}<\infty\,.
$$
The last integral can be estimated, using the change $w=\psi_1^{-1}(t)$ and the fact that $a_1(s)\leq 1$,
in the following way:
$$
\int_{\psi_1^{-1}(t_0)}^{0}\frac{ \psi_1'(w)dw}{\sqrt{\int_{w}^{1}e^{-2k_1(\tau)} d\tau}}
=
\int^{\psi_1^{-1}(t_0)}_{0}\frac{e^{-k_1(w)}a_1^{\frac{\beta}{\alpha}}(w)dw}{\sqrt{\int_{w}^{1}e^{-2k_1(\tau)}d\tau}}
\leq
\int^{\psi_1^{-1}(t_0)}_{0}\frac{dw}{\sqrt{\int_{w}^{w_0}e^{2[k_1(w)-k_1(\tau)]}d\tau}}
$$
where $w_0$ is chosen in such a way that $k_1'$ is decreasing in $(0,w_0]$.
We observe that $\displaystyle \int_0^1\sqrt{k_1'(t)}dt<\infty$, as $\theta<2$. Hence
it suffices to prove that there exists a strictly positive constant $c$ such that
$$
k_1'(w)\int_{w}^{w_0}e^{2[k_1(w)-k_1(\tau)]} d\tau\geq c\,.
$$
Now, since $k_1'(\tau)$ is decreasing in $(0,w_0]$,
$$
k_1'(w)\int_{w}^{w_0}e^{2[k_1(w)-k_1(\tau)]} d\tau\geq 
\int_{w}^{w_0}k_1'(\tau)e^{2[k_1(w)-k_1(\tau)]} d\tau
=\frac 12-\frac 12 e^{2[k_1(w)-k_1(w_0)]}\,.
$$
Observe that $\displaystyle e^{2k_1(w)}\to 0$ as $w\to 0$, since 
$\displaystyle k_1(w)=\int_1^{w}\frac{B}{\alpha t^{\theta}a_1(t)}dt\to -\infty$ as $w\to 0$, by hypothesis
$\theta\geq 1$. Therefore (\ref{condition_integrale}) is proved.

{\it Step 4.}
Theorem \ref{K-O} applies and gives, for every 
$\omega\subset\subset \Omega$, the existence of  a constant $C_{\omega}>0$  such that $v_n\leq C_{\omega}$. 
Recalling that
$\psi_n(s)\geq \psi_1(s)$ by (\ref{disu_psi_n}), we have
$C_{\omega}\geq v_n=\psi_n(u_n)\geq \psi_1(u_n)$.
Since  $\psi_1$ is decreasing,
$u_n\geq \psi_1^{-1}(C_{\omega})=c_{\omega}>0$ in every $\omega\subset \subset \Omega$.
\end{proof}

\begin{corollary}\label{secondo_corollario}
Let $1\leq \theta<2$. Let $u_n$ and
$u$ be
as in Lemma
\ref{lemma_convergenza_gradienti}. Then $\displaystyle \frac{|\nabla u|^2}{u^{\theta}} \in L^1(\Omega)$.
\end{corollary}
\begin{proof}
As in the proof of  Corollary \ref{primo_corollario}, we pass to the limit in
(\ref{primastima})
using the a.e. convergence of $u_n$ to $u$ (see Lemmata \ref{primolemmaaprioriestimates},
\ref{secondolemmaaprioriestimates} 
and
\ref{lemmastessobdo}), the a.e. convergence of $\nabla u_n$ to $\nabla u$ (see Lemma \ref{lemma_convergenza_gradienti})
and Proposition \ref{proposition_ko}.
 \end{proof}
\begin{corollary}\label{corollario}
For every $\omega \subset \subset \Omega$ 
there exists a positive constant $\tilde{c}_{\omega}$ such that
$$
\frac{u_n}{(u_n+\frac 1n)^{1+\theta}}  \leq \tilde{c}_{\omega} \,\,\,\forall\, x \in \omega\,.
$$
\end{corollary}
\begin{proof}
It is sufficient to observe that in every subset $\omega\subset \subset \Omega$
$$
\frac{u_n}{(u_n+\frac 1n)^{1+\theta}}\leq \frac{1}{u_n^\theta}\leq \frac{1}{c_{\omega}^\theta} = \tilde{c}_{\omega}\,,
$$
since $u_n\geq c_{\omega}>0$ in $\omega$ by Proposition \ref{proposition_ko}.
\end{proof}

As in \cite{luigietaltri} we prove the strong convergence of $T_k(u_n)$  in $H^1_{loc}(\Omega)$. 
This will be used to compute the limit of the lower order term in problems (\ref{pa}).
\begin{lemma}\label{ultimolemma}
Let $u_n$
be the solutions to problems (\ref{pa}) and $u$ be the function found in Lemmata \ref{primolemmaaprioriestimates}, \ref{secondolemmaaprioriestimates}, \ref{lemmastessobdo}.
Then, up to a subsequence, $T_k(u_n)\to T_k(u)$ in $H^1_{loc}(\Omega)$.
\end{lemma}
\begin{proof}
We are going to prove
that
$$
\lim_{n\to \infty}
\io|\nabla (T_k(u_n)-T_k(u))|^2\phi=0
$$
for all positive $\phi \in C_c^{\infty}(\Omega)$.
Let $\varphi_\lambda(s)=se^{\lambda s^2}$, $\lambda>0$.
As in \cite{bmp}, we will consider as a test function 
$
\varphi_\lambda(T_k(u_n)-T_k(u))\phi
$, 
where $\lambda$ will be chosen later. 
In the sequel  $\varepsilon(n)$ will denote any quantity converging to 0, as $n\to \infty$.
From (\ref{pa}) we get
\begin{equation}\label{lemma_ultimo}
\begin{array}{c}
\displaystyle\io \frac{b(x)}{(1+T_n(u_n))^p}\nabla u_n\cdot \nabla(T_k(u_n)-T_k(u))\varphi_\lambda'(T_k(u_n)-T_k(u))\phi
\\
\displaystyle
+
B \io\frac{u_n |\nabla u_n|^2}{(u_n+\frac 1n)^{\theta +1}}\varphi_\lambda(T_k(u_n)-T_k(u))\phi
\\
\displaystyle
=-\io \frac{b(x)}{(1+T_n(u_n))^p}\nabla u_n\cdot \nabla \phi\,\varphi_\lambda(T_k(u_n)-T_k(u))
+\io T_n(f)\varphi_\lambda(T_k(u_n)-T_k(u))\phi
\,.
\end{array}
\end{equation}
It is not difficult to prove that 
$$
\io T_n(f)\varphi_\lambda(T_k(u_n)-T_k(u))\phi\to 0
\,,\,\,\,\io \frac{b(x)}{(1+T_n(u_n))^p}\nabla u_n\cdot \nabla \phi\,\varphi_\lambda(T_k(u_n)-T_k(u))\to 0\,,
$$
as $n\to \infty$. Indeed for the first limit one can use the Lebesgue Theorem.
For the second one it is sufficient to observe that
$\nabla u_n$ converges weakly in some Sobolev space
given by the statements of Lemmata \ref{primolemmaaprioriestimates},
 \ref{secondolemmaaprioriestimates} and \ref{lemmastessobdo}
and 
$\displaystyle \frac{b(x)}{(1+T_n(u_n))^p}\nabla \phi\,\varphi_\lambda(T_k(u_n)-T_k(u))$ is uniformly bounded
with respect to $n$.

We are going to treat the left hand side of (\ref{lemma_ultimo}).
We choose $\omega_{\phi}\subset \subset \Omega$, with $\mbox{supp} \phi\subset \omega_{\phi}$. Then
$$
B \io\frac{u_n |\nabla u_n|^2}{(u_n+\frac 1n)^{\theta +1}}\varphi_\lambda(T_k(u_n)-T_k(u))\phi
\geq
-B\tilde{c}_{\omega_{\phi}}\io |\nabla T_k(u_n)|^2|\varphi_{\lambda}(T_k(u_n)-T_k(u))|\phi
$$
by Corollary \ref{corollario}.
We deduce from (\ref{lemma_ultimo}) that
\begin{equation}\label{lemma_ultimo_seconda}
\begin{array}{c}
\displaystyle\io \frac{b(x)}{(1+T_n(u_n))^p}\nabla u_n\cdot \nabla(T_k(u_n)-T_k(u))\varphi_\lambda'(T_k(u_n)-T_k(u))\phi
\\
\displaystyle
-B\tilde{c}_{\omega_{\phi}}\io |\nabla T_k(u_n)|^2|\varphi_{\lambda}(T_k(u_n)-T_k(u))|\phi\leq \varepsilon(n)\,.
\end{array}
\end{equation}
We remark that
$$
\int\limits_{\{u_n\geq k\}} \frac{b(x)}{(1+T_n(u_n))^p}\nabla u_n\cdot \nabla(T_k(u_n)-T_k(u))\varphi_\lambda'(T_k(u_n)-T_k(u))\phi=\varepsilon(n)\,.
$$
Hence  inequality (\ref{lemma_ultimo_seconda}) is equivalent to
\begin{equation}\label{lemma_ultimo_seconda_bis}
\begin{array}{c}
\displaystyle
\int\limits_{\{u_n\leq k\}} \frac{b(x)}{(1+T_n(u_n))^p}\nabla u_n\cdot \nabla(T_k(u_n)-T_k(u))\varphi_\lambda'(T_k(u_n)-T_k(u))\phi
\\
\displaystyle
-B\tilde{c}_{\omega_{\phi}}\io |\nabla T_k(u_n)|^2|\varphi_{\lambda}(T_k(u_n)-T_k(u))|\phi\leq \varepsilon(n)\,.
\end{array}
\end{equation}
Remark that
$$
\int\limits_{\{u_n\leq k\}} \frac{b(x)}{(1+T_n(u_n))^p}\nabla T_k(u)\cdot \nabla(T_k(u_n)-T_k(u))\varphi_\lambda'(T_k(u_n)-T_k(u))\phi\to 0\,,\,\,\, n\to \infty\,.
$$
Adding the above quantity in both sides of (\ref{lemma_ultimo_seconda_bis})
we get
$$
\int\limits_{\{u_n\leq k\}} \frac{b(x)}{(1+T_n(u_n))^p}\nabla (u_n-T_k(u))\cdot \nabla(T_k(u_n)-T_k(u))\varphi_\lambda'(T_k(u_n)-T_k(u))\phi
$$
$$
-B\tilde{c}_{\omega_{\phi}}\io |\nabla T_k(u_n)|^2|\varphi_{\lambda}(T_k(u_n)-T_k(u))|\phi
\leq \varepsilon(n)\,.
$$
By  hypothesis (\ref{hyp_a}) on $b$, we obtain
\begin{equation}\label{lemma_ultimo_terza}
\begin{array}{c}
\displaystyle
\int\limits_{\{u_n\leq k\}} \frac{\alpha}{(1+k)^p}|\nabla(T_k(u_n)-T_k(u))|^2\varphi_\lambda'(T_k(u_n)-T_k(u))\phi
\\
\displaystyle
-B\tilde{c}_{\omega_{\phi}}\io |\nabla T_k(u_n)|^2|\varphi_{\lambda}(T_k(u_n)-T_k(u))|\phi\leq \varepsilon(n)\,.
\end{array}
\end{equation}
It is easy to prove that
$$
\io |\nabla T_k(u_n)|^2|\varphi_{\lambda}(T_k(u_n)-T_k(u))|\phi\leq
\int\limits_{\{u_n\leq k\}}2|\nabla (T_k(u_n)-T_k(u))|^2|\varphi_\lambda(T_k(u_n)-T_k(u))|\phi+\varepsilon(n)\,.
$$
We deduce from (\ref{lemma_ultimo_terza}) that the quantity
\begin{equation}\label{ultimissima}
\displaystyle
\int\limits_{\{u_n\leq k\}} \left[\frac{\alpha}{(1+k)^p}\varphi_\lambda'(T_k(u_n)-T_k(u))-2B\tilde{c}_{\omega_{\phi}}|\varphi_\lambda(T_k(u_n)-T_k(u))|\right]|\nabla(T_k(u_n)-T_k(u))|^2\phi
\end{equation}
tends to 0.
Now, $\varphi_\lambda$ has the following property:
for every $a, b>0$,
$\displaystyle a\varphi_\lambda'(s)-b|\varphi_\lambda(s)|\geq \frac a2$ if $\displaystyle \lambda>\frac{b^2}{4a^2}$.
Therefore there exists  $\lambda>0$ such that
$$
\frac{\alpha}{(1+k)^p}\varphi_\lambda'(s)-2B\tilde{c}_{\omega_{\phi}}|\varphi_\lambda(s)|\geq \frac{\alpha}{2(1+k)^p} \,\,\,\,\forall\,s \in \R\,.
$$
Applying this inequality to the quantity
(\ref{ultimissima}),
the statement of the theorem is  proved.
\end{proof}

We are now going to prove Theorems \ref{teorema_energia_finita2} and \ref{teorema_energia_infinita2} in a unique proof.
As we will see the only difference is the choice of the test functions $\varphi$.  
Theorem \ref{teoremastimebdo} can be proved with the same technique.
\begin{proof}
By Lemmata \ref{primolemmaaprioriestimates} and
\ref{secondolemmaaprioriestimates}
the solutions $u_n$ to (\ref{pa}) are uniformly bounded in $H^1_0(\Omega)$ and $W^{1,\sigma}_0(\Omega)$ respectively; moreover
$\nabla u_n$ converges  to $\nabla u$ a.e. in $\Omega$ up to a subsequence, by Lemma \ref{lemma_convergenza_gradienti}. 
The solutions $u_n$ satisfy
$$
\io \frac{b(x)}{(1+T_n(u_n))^p}\nabla u_n\cdot\nabla \varphi
+B\io \frac{|\nabla u_n|^2 u_n}{(u_n+\frac 1n)^{1+\theta}}\varphi=
\io T_n(f)\varphi\,.
$$
For the proof of Theorem \ref{teorema_energia_finita2} we consider
for $\varphi$  a bounded  $H^1_0(\Omega)$ function.
For the proof of Theorem  \ref{teorema_energia_infinita2},
 $\varphi$ is a $C^1_0(\Omega)$ function. 
To compute 
the limit of the first term  
in the case where $u_n$ weakly converges to $u$  in $H^1_0(\Omega)$   (Theorem \ref{teorema_energia_finita2})
it is sufficient to use that
$\displaystyle \frac{b(x)}{(1+T_n(u_n))^p}\nabla \varphi$ strongly converges to $\displaystyle \frac{b(x)}{(1+u)^p}\nabla \varphi$
in $(L^2(\Omega))^N$ for every $\varphi \in H^1_0(\Omega)\cap L^{\infty}(\Omega)$.
In the case where
$u_n$ weakly converges  to $u$  in $W^{1,\sigma}_0(\Omega)$, with $\sigma<2$ (Theorem \ref{teorema_energia_infinita2}), one uses
that $\displaystyle \frac{b(x)}{(1+T_n(u_n))^p}\nabla \varphi$ strongly converges to $\displaystyle \frac{b(x)}{(1+u)^p}\nabla \varphi$
in $(L^r(\Omega))^N$ for every $r\geq 1$ and for every $\varphi \in C^1_0(\Omega)$.

To compute the limit of $\displaystyle \io \frac{|\nabla u_n|^2 u_n}{(u_n+\frac 1n)^{1+\theta}}\varphi$ 
we will use the same technique as in \cite{luigietaltri}. 
We are going to prove that 
$\displaystyle \frac{|\nabla u_n|^2 u_n}{(u_n+\frac 1n)^{1+\theta}}$ is equi-integrable.
Let $E \subset \subset \omega\subset \subset \Omega$. Then
$$
\int\limits_E \frac{|\nabla u_n|^2 u_n}{(u_n+\frac 1n)^{1+\theta}}\leq
\int\limits_{E\cap \{u_n\leq k\}} \frac{|\nabla u_n|^2 u_n}{(u_n+\frac 1n)^{1+\theta}}+
\int\limits_{E\cap \{u_n\geq k\}} \frac{|\nabla u_n|^2 u_n}{(u_n+\frac 1n)^{1+\theta}}
$$
$$
\leq
\tilde{c}_{\omega}\int\limits_{E\cap \{u_n\leq k\}} |\nabla T_k(u_n)|^2+
\int\limits_{\{u_n\geq k\}}\frac{|\nabla u_n|^2 u_n}{(u_n+\frac 1n)^{1+\theta}}\,,
$$ 
where we have used Corollary \ref{corollario} to estimate the first term.
Now, if we choose $T_1(u_n-T_{k-1}(u_n))$ in problems (\ref{pa}) we have, dropping the non-negative operator term,
\begin{equation}\label{ultima_passaggio_limite}
B\int\limits_{\{u_n\geq k\}}
\frac{|\nabla u_n|^2 u_n}{(u_n+\frac 1n)^{\theta +1}}
\leq
\int\limits_{\{u_n\geq k-1\}}f\,.
\end{equation}
Observe that there exists a constant $C>0$ such that
$\displaystyle \mu(\{u_n\geq k-1\})\leq \frac{C}{k-1}$, as $u_n$ are uniformly bounded in $L^1(\Omega)$.
This implies that the right hand side of (\ref{ultima_passaggio_limite})  converges to 0 as $k\to \infty$, 
uniformly with respect to $n$. We deduce that
there exists $k_0>1$ such that
\begin{equation}\label{primadispassaggiolimite}
\int\limits_{\{u_n\geq k\}}
\frac{|\nabla u_n|^2 u_n}{(u_n+\frac 1n)^{\theta +1}}\leq\frac{\varepsilon}{2}\,\,\,\,\,\,\,\,\forall\, k\geq k_0\,,\,\,\,\forall\, n\in \N\,.
\end{equation}
Moreover, since $T_k(u_n)\to T_k(u)$ in $H^1_{loc}(\Omega)$ by Lemma \ref{ultimolemma},
 there exist $n_{\varepsilon}, \delta_{\varepsilon}$ such that for every $E\subset\subset \Omega$
with $\mu(E)<\delta_{\varepsilon}$ we have
$$
\int\limits_{E\cap \{u_n\leq k\}} |\nabla T_k(u_n)|^2
=
\int\limits_{E} |\nabla T_k(u_n)|^2\leq
\frac{\varepsilon}{2 \tilde{c}_{\omega}}\,\,\,\,\,\,\,\forall\, n\geq n_{\varepsilon}\,.
$$
This and (\ref{primadispassaggiolimite}) imply that 
$\displaystyle \frac{|\nabla u_n|^2 u_n}{(u_n+\frac 1n)^{1+\theta}}$ is equi-integrable.
Now, recall that $\displaystyle \frac{|\nabla u_n|^2 u_n}{(u_n+\frac 1n)^{1+\theta}}$ converges a.e. to $\displaystyle \frac{|\nabla u|^2}{u^{\theta}}$, belonging to $L^1(\Omega)$ by Corollary \ref{secondo_corollario}.
By Vitali's theorem 
we have the result.
\end{proof}

\section{A non-existence result in the case $\theta \geq 2$}
We are going to prove Theorem \ref{thmnonexistence} about the non-existence of finite energy solutions to problem (\ref{pb})
when $\theta\geq 2$.
We will use the following result of \cite{luigietaltri}:
\begin{theorem}\label{pbellitticoausiliario}
Let $M(x,s)$ be a $N\times N$ matrix whose entries are Carath\'eodory functions $m_{ij}: \Omega \times \R\to \R$, for every $i, j=1,\dots, N$. Assume that
there exist two positive constants $\alpha_1, \beta_1$ such
that 
$
M(x,s)\xi\cdot \xi\geq \alpha_1 |\xi|^2
$ 
and
$
|M(x,s)|\leq \beta_1 
$
for a.e. $x \in \Omega$, and for all $(s,\xi) \in \R\times \R^N$.
Let $g: \Omega \times (0,+\infty)\to \R^+$ be a Carath\'eodory function such that
for some constants $s_0, \Lambda>0$ and $\theta\geq 2$ it holds
$$
g(x,s)\geq \frac{\Lambda}{s^{\theta}} \,\,\,\,\,\forall\,\,s\in (0,s_0]\,.
$$
Let $f\geq 0$, $f \not\equiv 0$, be a $L^q(\Omega)$ function, with $q>\frac N2$.
If one of the following conditions holds:
\begin{enumerate}
\item $\theta >2$
\item $\theta =2$ and $\lambda_1(f)>\frac{\beta_1}{\Lambda\alpha_1}$,  
\end{enumerate} then there is no
$H^1_0(\Omega)$ solution to problem
$$
\begin{cases}
\displaystyle
-\div\left(M(x,u)\nabla u\right) + g(x,u){|\nabla u|^{2}} = f & \mbox{in $\Omega$,} \\
\hfill u = 0 \hfill & \mbox{on $\partial\Omega$.}
\end{cases}
$$
\end{theorem}
\begin{proof} (of Theorem \ref{thmnonexistence})
By the change of variables
 $$
v=\left\{
\begin{array}{ll}
\displaystyle \frac{1-(1+u)^{1-p}}{p-1}\,, & p\neq 1
\\
\displaystyle\ln(1+u)\,, & p= 1\,,
\end{array}
\right.
$$
problem (\ref{pb}) is equivalent to
\begin{equation}\label{pbellitticoausiliariomio}
\begin{cases}
\displaystyle
-\div\left(b(x)\nabla v\right) + Bg(v){|\nabla v|^{2}} = f & \mbox{in $\Omega$,} \\
\hfill v = 0 \hfill & \mbox{on $\partial\Omega$,}
\end{cases}
\end{equation}
with
$$
g(s)=\left\{
\begin{array}{ll}
\displaystyle\frac{[1-(p-1)s]^{\frac{2p}{1-p}}}{([1-(p-1)s]^{\frac{1}{1-p}}-1)^\theta}\,, & p\neq 1
\\
\displaystyle\frac{e^{2s}}{(e^s-1)^\theta}\,, & p= 1\,.
\end{array}
\right.
$$
It is easy to prove that 
$g(s)s^{\theta}\to 1$, as $s\to 0^+$. Hence for every fixed $0<\varepsilon<B$
there exists $s_{\varepsilon}>0$
such that
$Bg(s)\geq \frac{B-\varepsilon}{s^{\theta}}$ for every $s \in (0,s_{\varepsilon}]$.
Theorem \ref{pbellitticoausiliario}  therefore applies to problem (\ref{pbellitticoausiliariomio}). 
We deduce that there is no $H^1_0(\Omega)$ solution
to problem (\ref{pbellitticoausiliariomio}) if either $\theta>2$, or $\theta=2$ and $\lambda_1(f)>\frac{\beta}{(B-\varepsilon)\alpha}$, for every $0<\varepsilon<B$. As a consequence there is no
$H^1_0(\Omega)$ solution
to problem (\ref{pb}) if either $\theta>2$, or $\theta=2$ and $\lambda_1(f)>\frac{\beta}{B\alpha}$.
\end{proof}

\section*{Acknowledgments} 
The author thanks  L. Orsina and M.M. Porzio for very helpful discussions on the subject of this paper and the anonymous referee for his suggestion to include Theorem \ref{thmnonexistence} in the paper.
Part of this work was done during a visit  to   La Sapienza Universit\`a di Roma whose hospitality is gratefully  
acknowledged.


\medskip
Received xxxx 20xx; revised xxxx 20xx.
\medskip

\end{document}